\documentclass[11pt, reqno]{amsart}
\usepackage{amsmath}
\usepackage{amsfonts}
\usepackage{graphicx,color}
\usepackage{amssymb}
\usepackage{amsthm}
\usepackage{lscape}
\usepackage{enumitem}
\usepackage{multirow}
\usepackage{pdflscape}

\usepackage{comment}

\usepackage{graphicx}
\usepackage{subfig}
\usepackage[all]{xy}

\makeatletter
\newtheorem*{rep@theorem}{\rep@title}
\newcommand{\newreptheorem}[2]{
\newenvironment{rep#1}[1]{
\def\rep@title{#2 \ref{##1}}
\begin{rep@theorem}}
{\end{rep@theorem}}}
\makeatother

\makeatletter
\def\subsubsection{\@startsection{subsubsection}{3}%
  \z@{.5\linespacing\@plus.7\linespacing}{.1\linespacing}%
  {\normalfont\itshape}}
\makeatother

\makeatletter
\def\subsection{\@startsection{subsection}{3}%
  \z@{.5\linespacing\@plus.7\linespacing}{.1\linespacing}%
  {\normalfont\bfseries}}
\makeatother

\theoremstyle{plain}
\newtheorem{thm}{Theorem}[section]
\newtheorem{cor}[thm]{Corollary}
\newtheorem{lem}[thm]{Lemma}
\newtheorem{prop}[thm]{Proposition}

\newreptheorem{theorem}{Theorem}
\newreptheorem{cor}{Corollary}

\theoremstyle{definition}
\newtheorem{de}[thm]{Definition}
 
\theoremstyle{remark}
\newtheorem{rem}[thm]{Remark}

\newcommand{\visom}{
\mathrel{\reflectbox{\rotatebox[origin=c]{90}{$\simeq$}}}}

\newcommand{\matrixofmap}{\mathrel{\begin{pmatrix}
  1 & \mapsto &(1,1) \\
\sigma &\mapsto &(x,-1)
 \end{pmatrix}}}

\newcommand{\qq}{\mathbb{Q}}
\newcommand{\ff}{\mathbb{F}}
\newcommand{\zz}{\mathbb{Z}}
\newcommand{\cc}{\mathbb{C}}

\newcommand{\pp}{\mathbb{P}}
\newcommand{\affine}{\mathbb{A}}
\newcommand{\rr}{\mathbb{R}}

\newcommand{\Spec}{\mathrm{Spec}}
\newcommand{\Proj}{\mathrm{Proj}}

\newcommand{\ri}[1]{\mathcal{O}_{#1}}

\newcommand{\fD}{\mathfrak{D}}
\newcommand{\fA}{\mathfrak{A}}
\newcommand{\deff}{{\buildrel\rm def\over=}}

\newcommand{\Rat}{\mathrm{Rat}}

\newcommand{\Aut}{\mathrm{Aut}}
\newcommand{\Gal}{\mathrm{Gal}}

\newcommand{\SLtwo}{\mathrm{SL_{2}}}
\newcommand{\pgltwo}{\mathrm{PGL_{2}}}
\newcommand{\GLtwo}{\mathrm{GL_{2}}}

\newcommand{\calC}{\mathcal{C}}
\newcommand{\calD}{\mathcal{D}}
\newcommand{\vt}{\vartheta}

\newcommand{\ti}[1]{\textit{#1}}

\newcommand{\bd}{\begin{de}}
\newcommand{\ed}{\end{de}}
\newcommand{\bl}{\begin{lem}}
\newcommand{\el}{\end{lem}}
\newcommand{\bp}{\begin{prop}}
\newcommand{\ep}{\end{prop}}
\newcommand{\bt}{\begin{thm}}
\newcommand{\et}{\end{thm}}
\newcommand{\bc}{\begin{cor}}
\newcommand{\ec}{\end{cor}}
\newcommand{\bpf}{\begin{proof}}
\newcommand{\epf}{\end{proof}}

\newcommand{\beq}{\begin{equation}}
\newcommand{\eeq}{\end{equation}}
\newcommand{\beqs}{\begin{equation*}}
\newcommand{\eeqs}{\end{equation*}}
\newcommand{\ben}{\begin{enumerate}}
\newcommand{\een}{\end{enumerate}}
\newcommand{\bit}{\begin{itemize}}
\newcommand{\eit}{\end{itemize}}

\title{The Moduli space of Cubic Rational Maps}

\author{Lloyd W. West}
\address{Ph.D Program in Mathematics, The Graduate Center, City University of New York; 365 Fifth Avenue,
New York, NY 10016 U.S.A. }
\email{lwest@gc.cuny.edu}
\thanks{2010 \emph{Mathematics Subject ClassiÞcation.} Primary: 37P45; Secondary: 13A50, 14D22}


\date{\today}

\begin{document}

\begin{abstract}
We construct the moduli space, $M_d$, of degree $d$ rational maps on $\pp^1$ in terms of invariants of binary forms. We apply this construction to give explicit invariants and equations for $M_3$. 

Using classical invariant theory, we give solutions to the following problems: (1) explicitly construct, from a moduli point $P\in M_d(k)$, a rational map $\phi$ with the given moduli; (2) find a model for $\phi$ over the field of definition (i.e. explicit descent). We work out the method in detail for the cases $d=2,3$. 
\end{abstract}

\maketitle

\section{Introduction}

Let $\phi:\pp^1\to\pp^1$ be a morphism of degree $d$ defined over a field $k$. The iteration of such maps has long been studied from the point of view of complex dynamics where $k=\cc$ (see \cite{Milbook}). More recently many interesting results and conjectures have been made concerning number theoretic aspects of their dynamics (see \cite{SilvermanBook}); that is, dynamics over fields $k$ that are not algebraically closed, especially extensions of $\qq$ and $\qq_p$.

\subsection{The Moduli Space $M_d$}

A rational map $\phi$ of degree $d$ defined over a field $k$, can be written $\phi(X_0,X_1)=[F_0(X_0,X_1):F_1(X_0,X_1)]$, where $F_0(X_0,X_1)$, $F_1(X_0,X_1)$ are homogenous forms of degree $d$ in $k[X_0,X_1]$ with no common factor in $\bar{k}[X_0,X_1]$.  Since the dynamical behavior of the maps is unchanged after performing the same change of coordinates on the source and target spaces, rational maps which are equivalent under the conjugation action of $\Aut(\pp^2)=\pgltwo$ are considered to be isomorphic.

The coarse moduli space, denoted $M_d$, of degree-$d$ rational maps up to conjugacy exists as a scheme over $\zz$ \cite{SilvermanSpace}. For $M_2$ there is an explicit description, due to Milnor \cite{Milnorquad}: namely, there is an isomorphism 
$$
(\sigma_1,\sigma_2): M_2 \to \affine^2_\cc
$$
given by the invariants $\sigma_1$ and $\sigma_2$, the first two symmetric functions in the multipliers of the fixed points of a map. Moreover Silverman \cite{SilvermanSpace} proved that one has an isomorphism $M_2\overset\sim\to\Spec \, \zz[\sigma_1,\sigma_2]\simeq\affine^2_\zz$ as schemes over $\zz$. In other words, over any field the ring of absolute invariants of quadratic rational maps is generated by $\sigma_1$ and $\sigma_2$. In general, it has been shown that $M_d$ is a rational variety for any $d$ \cite{Levy}.

In this paper we construct $M_d$ as an $\SLtwo$-quotient of a space of pairs of binary forms. Such quotients are well studied and there are methods for explicitly computing the ring of invariants. We give an explicit description of $M_3$ in these terms (Theorem \ref{M3}). Then, in sections \ref{construction} and \ref{evendegs}, we show how the inverse problem of constructing a rational map with a given set of invariants, i.e. a map that corresponds to a given point in the moduli space, can also be addressed using classical invariant theory. This method allows us to find the field of definition of a map in terms of its invariants.  

\subsection{Field of Moduli and Fields of Definition}

Fix a base field $k$.  For a rational map $\phi$ of degree $d$ with coefficients in $\bar{k}$, write $\phi^\gamma(X_0,X_1)\,\deff \,\gamma^{-1}\circ\phi\circ\gamma (X_0,X_1)$ for the conjugation action of $\gamma\in \pgltwo(\bar{k})$. Write $[\phi]$ for the point of $M_d(\bar{k})$ corresponding to the equivalence class of maps conjugate to $\phi$. Given a $k$-point $P\in M_d(k)$, there always exists a rational map $\phi$ with coefficients in $\bar{k}$ such that $[\phi]=P$; we say that $\phi$ is a \emph{model} for $P$. In sections \ref{construction} and \ref{evendegs} we present a method (in the odd and even degree cases respectively) for explicitly constructing such a model from the coordinates of the point $P$. 

When $k$ is not algebraically closed, a $k$-point $P\in M_d(k)$ may fail to have a model with coefficients in $k$. Recall the following standard definitions.

\bd\label{defFOM}(\emph{Field of Definition} and \emph{Field of Moduli})

\begin{enumerate} 
\item One says that a field $L$ is a \emph{field of definition} (FOD) for $\phi$ if there exists $\gamma\in \pgltwo(\bar{k})$ such that the coefficients of $\phi^\gamma$ are in $L$. We shall also say that $L$ is a field of definition for a moduli point $P$ if there exists a rational map $\phi$ defined over $L$ such that $[\phi]=P$. 
 \item \noindent Define 
$$
G_\phi:= \{\sigma\in \Gal(\bar{k}/k)\,:\,\sigma(\phi)=\phi^{\gamma_\sigma} \text{ for some } \gamma_\sigma \in \pgltwo(\bar{k})\}
$$
Then the \emph{field of moduli} (FOM) of $\phi$, denoted $k_\phi$, is the fixed field $\bar{k}^{G_\phi}$ of $G_\phi$. One has $[\phi]\in M_d(k_\phi)$.
\end{enumerate} 
\ed
For $d$ even, Silverman \cite{SilvermanFOD} has shown that the FOM is always equal to the FOD. Given a point $P\in M_d(k)$ ($d$ even), one should therefore be able to find a map $\phi$ defined over $k$ such that $[\phi]=P$. In section \ref{evendegs} we give a method for finding such a model that works for a generic odd degree map, and illustrate it explicitly in the case of quadratic maps. (Note that Manes and Yasufuku \cite{ManesYasufuku} have already given an explicit description of models, as well as twists, in the quadratic case.) 
 
When $d$ is odd, the FOD may in general be larger than the FOM. For example, Silverman \cite{SilvermanFOD} notes that 
\[
\psi(x)=i\left(\frac{x-1}{x+1}\right),
\]
has field of definition $\qq(i)$, but field of moduli $\qq$. 

\subsubsection{Cohomological Obstruction}

The obstruction to equality of FOM and FOD can be described cohomologically. This approach was taken in \cite{SilvermanFOD}. Similar methods were applied by Cadoret in \cite{Cadoret} to analyze the corresponding question for Hurwitz moduli spaces.  We first need some results on the automorphism groups of rational maps .

An \emph{automorphism} of a rational map $\phi$  of degree $d\geq 2$ is an element $\gamma\in\mathrm{PGL_2}(\bar{k})$ such that $\phi= \phi^\gamma $. 

The group of all automorphisms of the rational map $\phi$ shall be denoted $\Aut(\phi)$. This group is finite and its order is bounded in terms of $d$ (\cite{SilvermanBook} Prop. 4.65). 

The finite subgroups of $\pgltwo$ are well known:
\begin{prop}\label{Automslist}  

(1) Any  finite subgroup of $\pgltwo (\bar{k})$ is isomorphic to one of the following:
a cyclic group $\frak{C}_r$ of order $r$, a dihedral group $\fD_{2r}$ of order $2r$, the alternating group $\mathfrak{A}_4$, the symmetric group $\mathfrak{S}_4$, or the alternating group $\mathfrak{A}_5$.

(2) Any two finite subgroups of $\pgltwo (\bar{k})$ that are abstractly isomorphic are conjugate in $\pgltwo (\bar{k})$.

(3) Any finite subgroup $A<\mathrm{PGL_2}(\bar{k})$ is conjugate to a group $A_0$ for which both $A_0$ and its normalizer $N(A_0)$ are $G_k$-stable; hence these groups have the structure of a $G_k$-module, as does $Q(A_0)=N(A_0)/A_0$. (A list of such conjugacy class representatives $A_0$ along with $N$ and $Q$ can be found in \cite{Cadoret} Lemma 2.1). 
\end{prop}

Let $\phi$ be a model over $\bar{k}$ for $P\in M_d(k)$. After conjugation we may assume that $A=\Aut(\phi)$ is $G_k$-stable. From definition \ref{defFOM}, for each $\sigma \in G_k:=\Gal(\bar{k}/k)$ we have an element $\gamma_\sigma\in \pgltwo(\bar{k})$ such that 
\[\sigma(\phi)=\phi^{\gamma_\sigma}
\]
The assumption that $A$ is $G_k$-stable has the following consequence:
\bl (\cite{SilvermanFOD} Lemma 4.2)
\begin{itemize}
\item $\gamma_\sigma\in N(A)$
\item The map 
\begin{eqnarray*}
G_k &\longrightarrow &Q(A)\\
\sigma & \longmapsto & \gamma_\sigma
\end{eqnarray*}
gives a well defined element of $H^1(G_k, Q(A))$, which we shall denote $c_\phi$.
\end{itemize}
\el

From the quotient and inclusion morphisms $N(A)\to Q(A)$ and $N(A)\hookrightarrow \pgltwo(\bar{k})$, we get maps

   \begin{displaymath} 
\xymatrix{ 
 H^1(G_k, N(A))\ar@{->}[d]^{p} \ar@{->}[r]^{i} &H^1(G_k, \pgltwo(\bar{k})) \\ 
H^1(G_k, Q(A))&    }
\end{displaymath}

The \emph{cohomological obstruction} is defined to be the (possibly empty) set $I_k(\phi):=i(p^{-1}(c_\phi))$.

\bp The field $k$ is a field of definition for $\phi$ if and only if $I_k(\phi)$ contains the trivial class.
 \ep
\bpf
(Compare \cite{Cadoret} Prop. 2.4 and \cite{SilvermanFOD} Prop. 4.5)Suppose $\phi$ has a model $\psi$ defined over $k$. Then there exists $\eta\in \pgltwo(\bar{k})$ such that $\psi=\phi^\eta$. From  $\sigma(\psi)=\psi$, we have $\sigma(\phi)=\phi^{\eta\sigma(\eta)^{-1}}$ for all $\sigma\in G_k$. So the class $c_\phi$ can be represented by the cocycle $ \eta\sigma(\eta)^{-1}$, which clearly lifts to a true cocycle in $H^1(G_k, N(A))$ and is a $\pgltwo(\bar{k})$-coboundary.

The converse is covered by \cite{SilvermanFOD} Prop. 4.5.
\epf

When $p$ is surjective, the obstruction to FOM=FOD is given by a set of classes in $H^1(G_k, \pgltwo(\bar{k})) $. Such a class corresponds to a conic curve defined over $k$; the class is trivial precisely when the conic has a $k$-point. In particular, when $A=\{1\}$ the obstruction is given by a single conic curve, which we shall denote $\calC_P$. In section \ref{construction} we give a construction, in terms of the coordinates of the point $P$, of the conic $\calC_P$ in the case $d=3$. In the case where the obstruction vanishes (that is when the conic has a $k$-point), the construction furnishes an explicit model $\phi$ defined over $k$.  On a nonempty open set of $M_3$, the conic and explicit model can be constructed using identities in the ring of covariants associated to the $\SLtwo$-action, due to Clebsch \cite{Clebsch}. We shall call this the \emph{covariants method}. Such identities have been applied by Mestre and others (see \cite{Mestre}, \cite{LR}) to construct hyperelliptic curves of genus 2 and 3 from their moduli; in our case, we shall work with covariants of a system of two even degree binary forms (rather than a single even degree form, as in \cite{Mestre}), which complicates the procedure. To obtain explicit models of even degree rational maps (section \ref{evendegs}), we must deal with systems of binary forms of odd degree (a case that does not arise in \cite{Mestre}).

The covariants method fails on the closed subset, $M^\Aut_d$, of points with nontrivial automorphism group; here one must apply alternative \emph{ad hoc} constructions. In this paper we give explicit formulas only in the case $d=2$ and $d=3$, since for larger $d$, explicit generators for the ring of invariants are not known. However, the covariants method can be applied on a generic set of $M_d$ for any $d$, once the invariants are known.

\section{Invariants and Covariants}

In this section, we apply the classical invariant theory of binary forms to the construction of the moduli spaces $M_d$. We work over a field $k$ of characteristic zero.

Let $\phi: \pp^1_k \to \pp^1_k$ be a rational map. As above, we write $\phi$ in terms of coordinates $X_0,X_1$ on $\pp^1$ as $\phi(X_0,X_1)=[F_0(X_0,X_1):F_1(X_0,X_1)]$, where $F_0(X_0,X_1)$, $F_1(X_0,X_1)$ are homogenous forms of degree $d$ in $k[X_0,X_1]$ with no common factor in $\bar{k}[X_0,X_1]$. The forms $F_0$ and $F_1$ are specified by their coefficients:
\begin{eqnarray*}
F_0 (X_0,X_1)&=&c_1 X_0^d+c_2 X_0^{d-1}X_1+\cdots +c_{d}X_0X_1^{d-1}+c_{d+1} X_1^d\\
F_1 (X_0,X_1)&=&c_{d+2} X_0^d+c_{d+3} X_0^{d-1}X_1+\cdots +c_{2d+1} X_0X_1^{d-1}+c_{2d+2} X_1^d
\end{eqnarray*}
Such expressions can be parametrized by the projective space of the $2d+2$ coefficients, which I shall denote by $P_d=\mathrm{Proj}\, A_d$, where $A_d=k[c_1,\dots,c_{2d+2}]$. For a point of $P_d$ to represent a rational map of degree $d$, the polynomials $F_0$ and $F_1$ must have no common factor. This condition can be expressed by the non-vanishing of the resultant, $\mathrm{Res}(F_0, F_1)$, a polynomial in $A_d$ of degree $2d$, which I shall denote by $\rho_d$. The locus $\Rat_d=P_d - \{\rho_d=0\}$ is a fine moduli space for rational maps of degree $d$ in given coordinates. 

We recall briefly the representation theory of $\SLtwo$ over an algebraically closed field of characteristic zero. All representations are decomposable as direct sums of irreducible representations. Let $V$ be the space of linear binary forms in variables $X_0$ and $X_1$. The group $\SLtwo$ act by substitution 

\begin{equation}\label{action}
    \begin{array}{l r}

  \multirow{2}{*}{$\begin{pmatrix}
  p & q \\
 r &  s
 \end{pmatrix}$:}& X_0\mapsto p X_0+ q X_1\\
   &X_1\mapsto r X_0+ s X_1 \\
   \end{array}
\end{equation}

This action makes $V$ an irreducible representation of $\SLtwo$. Note that $V^*\simeq V$. Up to isomorphism, there is a unique irreducible representation of dimension $n+1$ for each integer $n\geq 0$; namely the space of homogenous binary forms of degree $n$ with the substitution action (\ref{action}); in other words, the $n$-th symmetric power, $S^nV$.

Thinking of the projective line as $\pp^1 =\pp V $, where $V=\Gamma (\pp^1, \ri{\pp^1 }(1))$, we can linearize the action of $\mathrm{PGL_2}$ on $\pp^1$ to the standard two dimensional representation of $ \SLtwo$ on $V$. Then $P_d \cong \pp W_d$, where $W_d=S^dV\otimes V^*\cong S^dV\otimes V$. Under this isomorphism, the map $\phi$ corresponds to the form 
\begin{equation}\label{hform}
h=Y_0F_0(X_0,X_1)+Y_1F_1(X_0,X_1).
\end{equation} 

The calculation of the Hilbert-Mumford criterion in \cite{SilvermanSpace} shows that $\Rat_d$ is contained in the stable locus of the linearized action of $\SLtwo$ on $P_d$, therefore we can form a geometric quotient 
 \[
 M_d=\Rat_d/\SLtwo = \Spec \, R_d
 \]
 where $R_d\,\deff \,H^0(\Rat_d,\ri{\Rat_d})^\SLtwo=((A_d[\frac{1}{\rho}])_0)^\SLtwo$. Since $\rho$ is itself an $\SLtwo$-invariant, we have $R_d\simeq(A_d)^\SLtwo[\frac{1}{\rho}]_0$, where the subscript `$0$' indicates the zero degree part of a graded ring. The ring $R_d$ is the subring of those functions in the coefficients of $\phi$ that remain unchanged under the conjugation action.
 
We have a compactification of $M_d$ by the embedding into 
 \[
 M_d^{ss}=\Proj \; (A_d^\SLtwo).
 \] 
 
\subsection{Expressing rationals maps in terms of binary forms} 

In this section we show how to find the ring of invariants $(A_d)^\SLtwo$ -- and hence find equations for the variety $M_d$ -- by recasting the problem into the classical invariant theory question of finding the simultaneous invariants of two binary forms.

\bp
There is an isomorphism of $\SLtwo$-modules
 \begin{eqnarray*}
\alpha: W_d&\xrightarrow[]{\quad\sim\quad}  &S^{d+1}V\oplus S^{d-1}V
\end{eqnarray*}
A rational map $\phi$, given by homogenous polynomials $F_0(X_0,X_1)$ and $F_1(X_0,X_1)$, corresponds, under $\alpha$, to a pair of binary forms $(f,g)$, where $g$ is the fixed point polynomial 
\[
X_1 F_0(X_0,X_1)- X_0 F_1(X_0,X_1);
\]
and $f$ is the divergence 
\[
\frac{\partial F_0}{\partial X_0}+\frac{\partial F_1}{\partial X_1}
\]
of the map $\mathbb{A}^2\to\mathbb{A}^2$ defined by $(F_0, F_1)$.
\ep

\bpf
One has the Clebsch-Gordan isomorphism of $\SLtwo$-modules $S^dV\otimes V\cong S^{d+1}V\oplus S^{d-1}V$. This is classical, but to obtain the explicit forms, $f$ and $g$, with fairly transparent dynamical significance, we shall modify the usual definitions of the operators in the Clebsch-Gordan series (as given, for example, in \cite{Procesi}). 

To explicitly realize the isomorphism, define the following operators 
\begin{eqnarray*}
\begin{pmatrix}
  \Delta_{xx} & \Delta_{xy} \\
  \Delta_{yx} & \Delta_{yy}
 \end{pmatrix}
 = \begin{pmatrix}
  X_0 & -X_1 \\
 Y_1 &  Y_0
 \end{pmatrix}
 \begin{pmatrix}
 \frac{ \partial}{\partial X_0} &\frac{ \partial}{\partial Y_1} \\
-\frac{ \partial}{\partial X_1} & \frac{ \partial}{\partial Y_0}
 \end{pmatrix}
\end{eqnarray*}
\begin{eqnarray*}
\tilde{\Omega}
 = \det
 \begin{pmatrix}
 \frac{ \partial}{\partial X_0} &\frac{ \partial}{\partial Y_1} \\
-\frac{ \partial}{\partial X_1} & \frac{ \partial}{\partial Y_0}
 \end{pmatrix}
\end{eqnarray*}
and write
\begin{eqnarray*}
 (x,y)
 = \det \begin{pmatrix}
  X_0 & -X_1 \\
 Y_1 &  Y_0
 \end{pmatrix}
\end{eqnarray*}

Then one has the following identity
\begin{eqnarray}\label{capelli}
(x,y)\tilde{\Omega}=  ( \Delta_{xx} +1) \Delta_{yy}-  \Delta_{yx}\Delta_{xy} 
\end{eqnarray}
If $h$ is a homogenous form of degree $m$ in $X_0$ and $X_1$, and degree $n$ in $Y_0$ and $Y_1$, then
\begin{eqnarray*}
\Delta_{xx} h&= &mh\\
\Delta_{yy} h&= &nh
\end{eqnarray*}
So applying the identity (\ref{capelli}) one has
\begin{eqnarray*}
h=\frac{1}{n(m+1)}[(x,y)\tilde{\Omega}  +\Delta_{yx}\Delta_{xy}]h
\end{eqnarray*}
This shows that the maps
\begin{eqnarray*}
\alpha: S^dV\otimes V&\to &S^{d+1}V\oplus S^{d-1}V\\
h &\mapsto& (\Delta_{xy}h,\tilde{\Omega} h)
\end{eqnarray*}
and 
\begin{eqnarray*}
\beta :S^{d+1}V\oplus S^{d-1}V &\to &S^dV\otimes V\\
(g,f)&\mapsto&\frac{1}{d+1}[ \Delta_{yx}g+(x,y)f]
\end{eqnarray*}
are mutually inverse. 

Given a rational map of degree $d$, written as $h=Y_0F_0(X_0,X_1)+Y_1F_1(X_0,X_1)$, we obtain a form $f$ of degree $d-1$ and a form $g$ of degree $d+1$:
\begin{equation}\label{fandg}
\begin{split}
f&=\tilde{\Omega} h \\
g&=\Delta_{xy} h
\end{split}
\end{equation}
It is a simple calculation, using the above-described explicit Clebsch-Gordon maps, to check that these correspond to the divergence and fixed point polynomials as claimed in the theorem statement.
\epf

From now on we shall refer to a rational map of degree $d$ as being specified either by the pair $(F_0,F_1)$ of homogenous forms of degree $d$ or by the pair $(f,g)$ of forms of degree $d-1$ and $d+1$. We shall refer to the moduli point in $M_d$ corresponding to a pair $(f,g)$ with the notation $[f,g]$.  

\begin{rem}
Let $P\in \pp^1(\bar{k})$ be a fixed point and choose coordinates $\xi_0,\xi_1$  for $P$ such that $F_i(\xi_0,\xi_1)=\xi_i$ for $i=0,1$. Then the multiplier of the fixed point $P$ is given by $f(\xi_0,\xi_1)-d$. 
\end{rem}
\begin{rem}
By work of Bogomolov and Katsylo, it is known that $\SLtwo$-quotients of spaces of binary forms are rational varieties. This gives an alternative means for proving the rationality of $M_d$.
\end{rem}

\subsection{Binary Forms, Invariants and Covariants}\label{binforms}

We briefly recall the basics of classical invariant theory (for a full account see the references \cite{Clebsch} and \cite{GraceYoung}).

Given an $\SLtwo$-variety, $W$, an $\SLtwo$-equivariant morphism $H: W\to S^eV$ is called a \emph{covariant} of order $e$. When $W$ is a linear representation of $\SLtwo$ we can write $W\cong \bigoplus_{1\leq \ell\leq n} V_{d_\ell}$ for some $n$ and $d_\ell\in \zz_{\geq 0}$. So a point of $W$ may be given as system of $n$ binary forms, $(f_\ell)_{1\leq \ell\leq n}$, where
\begin{equation*}
f_\ell=\sum_{i=0}^{d_\ell}c^{(\ell)}_i X_0^{d_1-i}X_1^{i}.
\end{equation*}

Elements $\begin{pmatrix}
  p & q \\
 r &  s
 \end{pmatrix}\in \SLtwo$ act on binary forms: if $f(X_0,X_1)$ is a form of degree $d$ with coefficients $\mathbf{c}=(c_1,\dots,c_d)$, the result of substituting $X_0=p X_0^\prime+ q X_1^\prime$ and $X_1=r X_0^\prime+ s X_1^\prime$ is a new form in $X_0^\prime$ and $ X_1^\prime$, whose coefficients we shall denote by $\mathbf{c^\prime}=(c_1^\prime,\dots,c_d^\prime)$.
 
Concretely, a \emph{covariant} of $W$ is given by a form 
\[
H((\mathbf{c}^{(\ell)})_{1\leq \ell\leq n},X_0,X_1)
\]
such that 
\begin{equation*}
H((\mathbf{c}^{\prime(\ell)})_{1\leq \ell\leq n},X^\prime_0,X^\prime_1)= H((\mathbf{c}^{(\ell)})_{1\leq \ell\leq n},X_0,X_1)
\end{equation*}
Note that the forms $f_\ell$ are themselves covariants. 

Such a covariant is homogenous in the two sets of variables $(X_0,X_1)$ and $(\mathbf{c}^{(\ell)})_{1\leq \ell\leq n}$ separately. The degree of $H$ in $(X_0,X_1)$ is the \emph{order} of the covariant.  The total degree in the variables $(\mathbf{c}^{(\ell)})_{1\leq \ell\leq n}$ shall be called simply the \emph{degree} of $H$. A covariant of order 0 is an \emph{invariant}. The quotient of two invariants of the same degree is called an \emph{absolute invariant}. 

The $\SLtwo$-covariants (resp. invariants) form a graded $k$-algebra $\mathrm{Cov}(W)$ (resp. $\mathrm{Inv}(W)$), which, by a famous result of Gordon and Hilbert, is always finitely generated; a set of generators is called a set of \emph{basic covariants} (resp. \emph{basic invariants}).

\subsection{Transvectants}


The \emph{omega process} with respect to variables $\mathbf{Z}^{(p)}=(Z_0^{(p)},Z^{(p)}_1)$ and  $\mathbf{Z}^{(q)}=(Z^{(q)}_0,Z^{(q)}_1)$ is defined as
\[
\Omega_{pq}=\frac{\partial^2}{\partial Z^{(p)}_0\partial Z^{(q)}_1}-\frac{\partial^2}{\partial Z^{(q)}_0\partial Z^{(p)}_1}
\]

Given two binary forms $F$ and $G$ in $X_0,X_1$ of orders $m$ and $n$ respectively, one defines the $r$-\emph{th transvectant} (also called the $r$-\emph{th \"uberschiebung}) as
\begin{eqnarray*}
 \scriptstyle (F,G)_r\;\deff\;\frac{(n-r)!(m-r)!}{n!m!} [(\Omega_{12})^r \{F(Z^{(1)}_0,Z^{(1)}_1)G(Z^{(2)}_0,Z^{(2)}_1)\}]\big{|}_{\normalsize X_0=Z^{(1)}_0=Z^{(2)}_0,\;X_1=Z^{(1)}_1=Z^{(2)}_1}
\end{eqnarray*}
\normalsize

If $F$ and $G$ are covariants of degree $d$ and $e$ respectively, then $(F,G)_k$ is a covariant of order $m+n-2k$ and degree $e+d$. 

In fact, any covariant of a system of binary forms can be expressed as a linear combination of iterated transvectants of the ground forms. This fact is known as the First Fundamental Theorem of Invariant Theory. 

We shall also use the \emph{generalized transvectant} (see \cite{GraceYoung} $\S$ 81). Given a sequence of distinct pairs $(p_1,q_1),\dots, (p_k,q_k)$, where $p_i,q_i\in \{1,\dots,m\}$ and $p_i\neq q_i$, define 
 \[
 \kappa_\ell  \; \deff \;\sum_{i\;:\; \ell\in\{p_i,q_i\}} r_i
 \]
Then one can define the generalized transvectant of the $m$ forms $G_1,\dots,G_m$ of orders $s_1,\dots,s_m$ as
\begin{equation*}
\begin{split}
&(G_{p_1},G_{q_1})_{r_1}(G_{p_2},G_{q_2})_{r_2}\cdots(G_{p_k},G_{q_k})_{r_k}\;\deff\; \\
&\prod_{\ell=1}^n \frac{(s_\ell-\kappa_\ell)!}{s_\ell !} \cdot \left[\prod_{i=1}^k (\Omega_{p_i q_i})^{r_i} \cdot \prod_{\ell=1}^{m} G(Z^{(\ell)}_0,Z^{(\ell)}_1)\right]\bigg|_{X_0=Z^{(\ell)}_0,\;X_1=Z^{(\ell)}_1 \; 1\leq\ell\leq m}
\end{split}
\end{equation*}
\newline

\subsection{Typical Presentations}\label{relations}
In the sequel we shall follow a method similar to that used by Mestre in \cite{Mestre}. This method depends on Clebsch's construction of `typical presentations' (\emph{`typische Darstellungen'}) for systems of binary forms (\cite{Clebsch} $\S$81).

\subsubsection{Systems where at least one of the forms is of odd order}\label{oddorder}
 
Let $f_1, \dots, f_n$ be a system of binary forms that contains at least one form of odd order. Then one can prove (\cite{Clebsch} $\S$90)
that there exist two linear covariants, $u_0$ and $u_1$ say, whose resultant does not vanish identically. 

Given a n arbitrary form of the system, $f$ say, write it in symbolic notation\footnote{`Symbolic notation' is a compact calculus for manipulating covariants of binary forms; see \cite{GraceYoung} or \cite{Clebsch} $\S$15 for a classical account.} as $f=a_x^m$.  Then from the basic symbolic identities one has
\begin{equation}\label{linrel}
(u_1u_0)_1 \cdot a_x=(u_1f) \cdot u_0 - (u_0f) \cdot u_1
\end{equation}
On taking the $m$-th symbolic power this yields
\begin{equation}\label{genrel}
\begin{split}
((u_1,u_0&)_1)^m \cdot f=\\
& \sum_{i_1,\dots i_m\in \ff_2}  (-1)^{i_1+\cdots+i_m} \cdot(u_{i_1},f)_1\cdots(u_{i_m},f)_1 \cdot  u_{i_1 +1}\cdots u_{i_m +1}
\end{split}
\end{equation}
Note that in this expression the products of transvectants should be read as symbols for generalized transvectants, not as the mere product of the simple transvectants.  

In (\ref{genrel}) we have expressed the form $f$ as a binary form in new variables $u_0$ and $u_1$ such that the coefficients are invariants and the change of variables is given by covariants; this is what is meant by `typical presentation' of the form $f$. 

 \subsubsection{Systems of forms of even order}\label{subsec:evenorder}

In the case that all of the original forms $f_\ell$ are of even order, one does not dispose of independent linear covariants. One can however prove (\cite{Clebsch} $\S$102) that there exists a pair of \emph{quadratic} covariants, $u_1$ and $u_2$ say, whose resultant does not vanish identically. Given these one can always find a third quadratic covariant $u_3$ such that all three are generically linearly independent (for example, one can take the first transvectant of $u_1$ and $u_2$).

The three covariants $u_1,u_2,u_3$ can be treated as a system of quadratic forms. In a system of quadratic forms one has three kinds of basic covariants; namely the covariants of order two
\begin{equation*}
\xi_i = (u_j, u_k )_1\quad \text{for cyclic permutations } (ijk) \text{ of } (123)
\end{equation*}
together with invariants
\begin{equation*}
C_{ij} = (u_i, u_j )_2 \quad \text{for } 1\leq i,j \leq 3
\end{equation*}
and
\begin{equation*}
r=-(u_1,u_2)_1 (u_1,u_3)_1 (u_2,u_3)_1
\end{equation*}
Note that $r$ is the determinant of $u_i$ in the basis $X_0^2, \,X_0X_1,\,X_1^2$. Therefore the forms $u_i$ are linearly independent if and only if $r\neq 0$. 

In (\cite{Clebsch} $\S$103) the following formulas are deduced, from which one can find a typical representative for the system.
If $F$ is any additional quadratic form one has
\begin{equation}\label{rf}
r F= (F,u_1)_2 \xi_1+(F,u_2)_2 \xi_2+(F,u_3)_2 \xi_3
\end{equation}
One obtains a similar relation for any form $G$ of even order $2n$, which we write using symbolic notation as
\begin{equation}\label{rg}
r^n G = ((G,u_1)_2 \xi_1+(G,u_2)_2 \xi_2+(G,u_3)_2 \xi_3)^{(n)}\end{equation}
Here the power `$^{(n)}$' should be interpreted symbolically; that is, products of transvectants should be read as symbols for generalized transvectants, not as the mere product of the simple transvectants.  If $G$ is one of the original forms $f_\ell$ of the system, then (\ref{rg}) gives us a typical presentation for that form. 

As in \cite{Clebsch} ($\S$58) one can prove various relations between the covariants. Firstly, one has 
\begin{equation}\label{uixi}
u_1\xi_1+u_2\xi_2+u_3\xi_3=0
\end{equation}
Applying (\ref{rf}) with $u_i$ substituted for $F$ and then substituting these expressions for $u_i$ into (\ref{uixi}), one obtains
\begin{equation}
\sum_{1\leq i,j\leq 3} C_{ij}\xi_i\xi_j = 0
\end{equation}
From this, one can find a relation between the invariants:
\begin{equation}\label{M3rel}
2r^2=\det (C_{ij})
\end{equation}
\begin{rem}
Proofs for the identities stated in this section can all be found in \cite{Clebsch} or \cite{LR}.
\end{rem}

\section{Moduli Space of Cubic Rational Maps}\label{Cubic}

\subsection{Invariants, Covariants and Relations}

The generators of the algebra of covariants for systems of forms of low degree were calculated in the nineteenth century. We now list the basic covariants in the case relevant to the moduli space of cubic maps; namely the system of one quadratic form $f$ and one quartic form $g$ . In this case there are 18 basic covariants, of which six are invariants (see \cite{GraceYoung} $\S$ 143, \cite{Clebsch} $\S$60). For ease of notation, define $H=(g,g)_2$ and $T=(g,H)_1$.

 \begin{table}[!h]
  \begin{small}
\begin{center}
    \begin{tabular}{| c||c |c | c| c|c|c |}
    
     \hline
     \multicolumn{1}{|r||}{\quad \quad Deg.}&\multirow{2}{*}{\quad1\quad} & \multirow{2}{*}{\quad 2 \quad}&\multirow{2}{*}{\quad 3 \quad} &\multirow{2}{*}{\quad 4 \quad}  &\multirow{2}{*}{\quad 5 \quad }  &\multirow{2}{*}{\quad 6 \quad}  \\ 
     \multicolumn{1}{| l ||}{Ord.\quad\quad} & & & & & & \\
   \hline \hline
    \multirow{2}{*}{0}&  \multirow{2}{*}{--} & $(g,g)_4$& $(H,g)_4$ & \multirow{2}{*}{$(H,f^2)_4$} & \multirow{2}{*}{--} &\multirow{2}{*}{$(T,f^3)_6$} \\

  & &$(f,f)_2$ &$(g,f^2)_4$& & &\\
  \hline
  \multirow{2}{*}{2}&  \multirow{2}{*}{$f$} & \multirow{2}{*}{$(g,f)_2$}& $(H,f)_2$ & \multirow{2}{*}{$(H,f^2)_3$} & \multirow{2}{*}{$(T,f^2)_4$} &\multirow{2}{*}{--} \\
  & & &$(g,f^2)_3$& & &\\
  \hline
  \multirow{2}{*}{4}&  \multirow{2}{*}{$g$} &$(g,f)_1$& \multirow{2}{*}{$(H,f)_1$} & \multirow{2}{*}{$(T,f)_2$} & \multirow{2}{*}{--} &\multirow{2}{*}{--} \\
  & &$(g,g)_2$ && & &\\
   \hline
  \multirow{2}{*}{6}&\multirow{2}{*}{--} &\multirow{2}{*}{--}&\multirow{2}{*}{$(g,H)_1$ }&\multirow{2}{*}{--} &\multirow{2}{*}{ --}&\multirow{2}{*}{--}\\
  & & & & & & \\
   \hline
   \end{tabular} 
   \end{center}\end{small}
   \caption{Covariants and Invariants for quadratic and quartic $(f,g)$}
  \label{tab:Cov}
\end{table} 

The six basic invariants  --  denoted $i,j,a,b,c,d$ -- are bihomogenous in the coefficients of $f$ and $g$ respectively. The multidegrees are recorded in Table \ref{tab:Invariants}.
\begin{center}
\begin{table}[h]
  \centering
    \begin{tabular}{ c  c  l  c c}
    \hline
    \multicolumn{3}{c}{Invariants}  & Degree & Multidegree \\ \hline
    $d$ & $=$ & $(f,f)_2$ & 2 & (2,0) \\ 
    $i$ & $=$ & $(g,g)_4$ & 2 & (0,2) \\ 
    $j$ & $=$ & $(H,g)_4$ & 3 & (0,3) \\ 
    $a$ & $=$ & $(g,f^2)_4$ & 3 & (2,1) \\ 
    $b$ & $=$ & $(H,f^2)_4$ & 4 & (2,2) \\ 
    $c$ & $=$ & $(T,f^3)_6$ & 6 & (3,3) \\ 
    \hline
    \end{tabular}
\caption{Invariants for quadratic and quartic $(f,g)$}\label{tab:Invariants}
\end{table}
\end{center}
For later use we define the following invariant 
\begin{equation}
\tilde{c}=\frac{1}{6} d^{2} i - \frac{1}{2} a^{2} + \frac{1}{2} d b
\end{equation}
The resultant $\rho$ of the forms $F_0$ and $F_1$ can be expressed in terms of the basic invariants as
\begin{equation}
\rho=\frac{1}{8}i^3 + \frac{1}{384} i d^2 - \frac{3}{4} j^2 - \frac{3}{16} j a + \frac{1}{256} a^2 + \frac{3}{16} i b - \frac{1}{64} d b - \frac{1}{8} c
\end{equation}

\begin{rem}
Multiplying all coefficients of a rational map $\phi$ by a nonzero factor $\alpha$ does not change the rational map; however the invariants of degree $n$ change by a factor of $\alpha^n$. Conversely any two rational maps $\phi$ and $\psi$ are conjugate over $\bar{k}$ if and only if there is an element $\alpha \in \bar{k}$ such that 
\[
(d_\phi, i_\phi,j_\phi,a_\phi,b_\phi,c_\phi)= (\alpha^2 d_\psi, \alpha^2 i_\psi, \alpha^3 j_\psi, \alpha^3 a_\psi, \alpha^4 b_\psi, \alpha^6 c_\psi);
\]
or equivalently if and only if $(d_\phi, i_\phi,j_\phi,a_\phi,b_\phi,c_\phi)$ and $(d_\psi, i_\psi,j_\psi,a_\psi,b_\psi,c_\psi)$ represent the same point of the weighted projective space $\pp(2,2,3,3,4,6)$, where, for example, $d_\phi$ means the value of the invariant $d$ evaluated at the coefficients of $\phi$.

\end{rem}

In the next section we shall use various covariants of $f$ and $g$. They are listed in Table \ref{tab:Covariants}.

\begin{table}[h]
  \centering
\begin{center}
    \begin{tabular}{ c  c  l  c c c}
    \hline
    \multicolumn{3}{c}{Covariants}& Order  & Degree & \\ \hline
    $H$ & $=$ & $(g,g)_2$ & 4 & 2&\\
    $T$ & $=$ & $(g,H)_1$ & 6 & 3&\\
    $u_1$ & $=$ & $f$ & 2 & 1 &\\  
    $u_2$ & $=$ & $(g,f)_2$ & 2 & 2& \\  
    $u_3$ & $=$ & $(H,f)_2$ & 2 & 3& \\ 
    $\xi_1$ & $=$ & $(u_2,u_3)_1$ & 2 & 9& \\ 
    $\xi_2$ & $=$ & $(u_3,u_1)_1$ & 2 & 8& \\ 
    $\xi_3$ & $=$ & $(u_1,u_2)_1$ & 2 & 7& \\ 
    $C_{ij}$ & $=$ & $(u_i,u_j)_2$ & 0 & $i+j$ & for $1\leq i,j \leq 3$\\ 
    $A_{i}$ & $=$ & $(f,u_i)_2$ & 0 & $1+i$ &  for $1\leq i \leq 3$\\  
    $B_{ij}$ & $=$ & $(g,u_i)_2 (g,u_j)_2$ & 0 & $1+i+j$ & for $1\leq i,j \leq 3$\\
    $r$ & $=$ & $-(u_1,u_2)_1 (u_1,u_3)_1 (u_2,u_3)_1$ & 0 & 6 & \\
    \hline
    \end{tabular}\end{center}\caption{Covariants for cubic rational map}\label{tab:Covariants}
\end{table}
\subsection{Relations}\label{relations}
There is a single relation among the six basic invariants of the cubic rational map; namely the relation
\begin{equation}\label{M3rel2}
2r^2=\det (C_{ij})
\end{equation}
coming from the relation (\ref{M3rel}) mentioned in section \ref{subsec:evenorder}.
One can calculate the $C_{ij}$ in terms of the basic invariants (see \cite{Clebsch} $\S$ 60):
\begin{eqnarray*}
\begin{array}{llllll} 
C_{11} = d ,&& \quad\quad C_{12}  =  a ,&& \quad\quad C_{13}  =  b\\
\\
C_{22} = b+\frac{1}{3}id ,&& \quad\quad C_{23}  =  \frac{1}{6}ia+\frac{1}{3}jd ,&& \quad\quad C_{33}  =  \frac{1}{3}ja-\frac{1}{6}ib+\frac{1}{18}i^2d
\end{array}
\end{eqnarray*}

Furthermore, one can check that $r=c$; so the relation (\ref{M3rel2}) becomes
\begin{equation}\label{relfull}
2 c^{2}=\frac{1}{54} d^{3} i^{3} - \frac{1}{9} d^{3} j^{2} - \frac{1}{12} d i^{2} a^{2} - \frac{1}{3} j a^{3} + d j a b + \frac{1}{2} i a^{2} b - \frac{1}{2} d i b^{2} -  b^{3}
\end{equation}

As a corollary of these classical invariant theory computations, we have the following description of $M_3$.
\bt\label{M3} 
The space $M_3$ is isomorphic to the 4-dimensional variety in $\pp(2,2,3,3,4,6)$ determined by the relation (\ref{relfull}) and the non-vanishing of the resultant $\rho$.
\et



\subsection{Automorphisms}\label{autos}

In this section we determine the locus of maps having automorphism group isomorphic to $A$, for each of the possible groups $A$ listed in Proposition \ref{Automslist}.

If a rational map $\phi$ corresponds to the pair of binary forms $f,\,g$ then, for each $\gamma \in \Aut(\phi)$, we must have
\begin{eqnarray*}
f\circ \gamma &=& \chi(\gamma) f\\
g\circ \gamma &= &\chi(\gamma) g
\end{eqnarray*}
for some $\chi(\gamma) \in \bar{k}^\times$. It is easy to see that $\chi: G\to \bar{k}^\times$ must  be a character of $G$. In other words $f$ and $g$ must be \emph{relative invariants} with the same character. Therefore, to find representatives of the $\bar{k}$-conjugacy classes of maps with a given automorphism group it is enough to find relative invariants. Any relative invariant is a polynomial in the so-called \emph{Grundformen}; that is forms whose set of zeros is equal to an exceptional orbit of the action of $G$ on $\pp^1(\bar{k})$ (i.e. an orbit with non-trivial stabilizer). (This approach is used in \cite{McMullen} to find maps with icosahedral automorphism group.)

By calculating Grundformen for each finite subgroup $A_0< \mathrm{PGL_2}(\bar{k})$ from Proposition \ref{Automslist}, we find normal forms for the maps whose automorphism group contains $A_0$; they are listed in Table \ref{tab:Auts}. Each such normal form corresponds to a locus in the moduli space $M_3$. We find a defining ideal for closure of this locus in $M_3$ by computing the invariants of the normal form in terms of the parameters $p_i$ and then eliminating the $p_i$. In fact the loci are themselves closed; this is a corollary of the construction of rational maps from their moduli given in section \ref{construction}. In this way, we have a stratification of the moduli space by automorphism group; the organization of the strata is given in Figure \ref{fig:strata}. Note that we do not list normal forms of binary forms that do not correspond to rational maps, i.e. those for which the resultant vanishes.  There are no cubic rational maps with automorphism group $\mathfrak{S}_4$ or $\mathfrak{S}_5$ since the Grundformen for these groups have degree at least 6 and 11 respectively. 

 \begin{table}[!h]
  \begin{small}
\begin{center}
    \begin{tabular}{ |c|c |c  c| c| c|}
    \hline
    
     \multirow{2}{*}{Group}&\multirow{2}{*}{Name} & \multicolumn{2}{ c| }{Normal Form} &\multirow{2}{*}{Ideal}  &\multirow{2}{*}{Dimension} \\ \cline{3-4}
     & &$f$ & $g$& & \\
   \hline \hline
    \multirow{2}{*}{$\zz/2\zz$}& $\mathsf{\mathsf{C_2^{(1)}}}$&$s X_0 X_1$ &$t X_0^4 +uX_0^2X_1^2+vX_1^4$& $ (c,\tilde{c})$ & 2\\       
     &$\mathsf{C_2^{(2)}}$&$s X_0^2+t X_1^2$ &$X_0 X_1(uX_0^2+v X_1^2)$&$(a,j)$ & 2 \\
    \hline
   $\zz/3\zz$& $\mathsf{C}_3$&$s X_0^2$ &$ X_1(tX_0^3+u X_1^3)$ &$(d,i,b)$& 1\\
   \hline
   $\zz/4\zz$& \multicolumn{5}{ c| }{$(\zz/4\zz)\leq \Aut(\phi)\Longrightarrow   \Aut(\phi)=\fD_8   $; see $\mathsf{D}_8$ below}\\
   \hline
   \multirow{2}{*}{$\fD_4$} & $\mathsf{D_4^{(1)}}$ & $s X_0X_1$ &$ t(X_0^4-X_1^4)$&$(a,j,c,\tilde{c})$& 1\\   
    &$\mathsf{D_4^{(2)}}$&$0$ &$ s(X_0^4+X_1^4)+tX_0^2X_1^2$&$(d,a,b,c)$&1\\
    \hline    
  $\fD_8$& $\mathsf{D}_8$&$0$ &$ s(X_0^4+X_1^4)$&$(d,j,a,b,c)$&0\\
  \hline
  $\mathfrak{A}_4$& $\mathsf{A}_4$&$0$ &$X_0^4-2\sqrt{-3}X_0^2X_1^2+X_1^4$ & $ (c,b,a,i,d)$&0  \\  
  \hline

   \end{tabular} 
   \end{center}\end{small}
   \caption{Loci with non-trivial automorphism group in $M_3$}
  \label{tab:Auts}
\end{table} 

\begin{figure}[h!]
  \centering
   \begin{displaymath} 
\xymatrix{ 
\text{dimension 2}& & & \mathsf{C_2^{(1)}}\ar@{-}[d] \ar@{-}[drr] & &  \mathsf{C_2^{(2)}}\ar@{-}[d]  \\ 
\text{dimension 1}& \mathsf{C}_3 \ar@{-}[dr]& & \mathsf{D_4^{(2)}}\ar@{-}[dl] \ar@{-}[dr] & &  \mathsf{D_4^{(1)}}\ar@{-}[dl]   \\ 
\text{dimension 0}& &\mathsf{A}_4 &  & \mathsf{D}_8& }
\end{displaymath} 

  \caption{Organization of the strata}
    \label{fig:strata}
\end{figure}
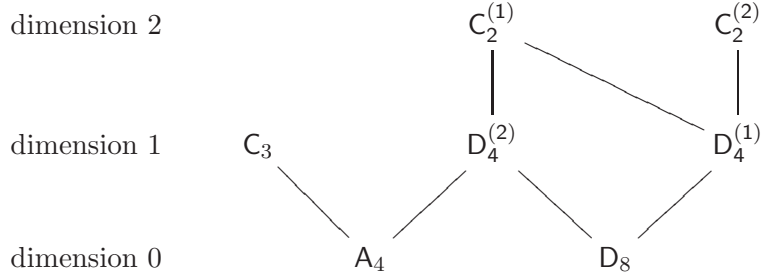

The singular locus is related to nontrivial automorphisms as follows.
\bp
The variety $M_3$ is singular. The locus of points in $M_3$ corresponding to maps with non-trivial automorphism group, which we shall denote $M_3^{\Aut}$, is equal to the singular locus $M_3^{sing}$. The locus $\mathsf{C_2^{(1)}}$ is precisely the singular locus of the affine quasi-cone over $M_3$. 
\ep
\bpf
This is easy to verify computationally from the explicit description of the invariant ring in terms of generators and relations.
\epf


\section{Constructing Cubic Rationals Maps from their Moduli}\label{construction}

Given a point $P\in M_3(k)$ there is always a rational map $\phi$ defined over $\bar{k}$ corresponding to the given point. In this section we explicitly construct such a model $\phi$ and investigate when it can be defined over $k$.

We shall see that, so long as $P$ is not in the locus $\mathsf{C_2^{(1)}}$, any model of $P$ has two independent quadratic covariants. In this case, the covariants formulas of section \ref{subsec:evenorder} afford a method for constructing a model $\phi$ with $[\phi]=P$. In  section \ref{indep} we explain this construction in full.

On the locus $\mathsf{C_2^{(1)}}$, one cannot use the covariants method. In section \ref{dep} we give a different construction that works in this case. 

\subsection{Weighted projective coordinates}

A point is $P\in M_3(\bar{k})$ is defined over $k$ if all the absolute invariants are in $k$, or equivalently if $P^\sigma=P$ for all $\sigma\in \Gal (\bar{k}/k)$. By the construction of $M_3$ in terms of relative invariants, we have an embedding $M_3\hookrightarrow \pp(2,2,3,3,4,6)$. Recall that a $\bar{k}$-point $P$ of a weighted projective space $\pp(w_0,\dots ,w_n)$ can be represented by tuples $(x_0,\dots x_n)\in \bar{k}^{n+1}$; another such tuple $(x^\prime_0,\dots x^\prime_n)$ represents the same point if and only if there exists $\alpha \in \mathbb{G}_m(\bar{k})$ such that $\alpha^{w_i} x_i=x^\prime_i$ for $0\leq i\leq n$. Such an $n+1$-tuple will be called a set of \emph{weighted projective coordinates} for the point $P$. We note that by Hilbert's Theorem 90, one can choose weighted projective coordinates with values in $k$ to represent any point $P\in\pp(w_0,\dots ,w_n)$. In particular this applies to points $P\in M_3(k)$. 

\begin{lem}\label{hilb90}
If $P\in M_3(k)$ is a point defined over $k$, then there exist values $\mathbf{d},\mathbf{i},\mathbf{j},\mathbf{a},\mathbf{b},\mathbf{c}\in k$ of the invariants such that $P=[\mathbf{d}:\mathbf{i}:\mathbf{j}:\mathbf{a}:\mathbf{b}:\mathbf{c}]$.
\end{lem}
\begin{proof}
Hilbert's Theorem 90.  
\end{proof}

\subsection{The covariants method}\label{indep}

We start with the following observation. 
\bp  A pair $(f,g)$ of forms of degree 2 and 4 respectively \emph{fails} to have two independent quadratic covariants if and only if both $c$ and $\tilde{c}$ vanish. 
\ep
\bpf
\cite{Clebsch} $\S\S$107-108.
\epf

In the case $c\neq0$ the relevant quadratic covariants were given in Table \ref{tab:Covariants}. The associated invariants, $A_i, B_{ij}$ and $C_{ij}$, were defined in section \ref{relations} and the expressions for $C_{ij}$ in terms of the basic invariants were listed. For completeness, we list here the expressions for  $A_i$ and $ B_{ij}$.
\begin{small}
\begin{eqnarray*}
\begin{array}{llllll} 
A_1= d,&& \quad A_2  =  a ,&& \quad A_3  =  b\\
B_{11} = a ,&& \quad B_{12}  =  b+\frac{1}{3}id ,&& \quad B_{13}  =  \frac{1}{6}ia+\frac{1}{3}jd\\
B_{22} = \frac{1}{2}ia+\frac{1}{3}jd,&& \quad B_{23}  =  \frac{1}{3}ja-\frac{1}{6}ib+\frac{1}{18}i^2d ,&& \quad B_{33}  =  \frac{1}{3}jb-\frac{1}{36}i^2a+\frac{1}{18}dij
\end{array}
\end{eqnarray*}
\end{small}

In the following theorem we use the identities from section \ref{subsec:evenorder} to construct a model $\phi$ corresponding to a given moduli point $P$; we also determine when FOM=FOD.

\begin{thm}\label{covariantstheorem}
Let $P\in M_3(k)$ be a moduli point outside of the locus \mbox{$\{c=0\}$.} Choose weighted projective coordinates \mbox{$P=[\mathbf{d}:\mathbf{i}:\mathbf{j}:\mathbf{a}:\mathbf{b}:\mathbf{c}]$} with values in $k$; whence we also have values $\mathbf{A}_i, \mathbf{B}_{ij}, \mathbf{C}_{ij} \in k$.

Let $\calC_P $ be the conic in $\pp^2$ with equation
 \begin{equation*}
\sum_{1\leq i,j\leq 3} \mathbf{C}_{ij} x_i x_j = 0
\end{equation*}

\noindent Then 
\begin{center}
$\calC_P(k)\neq\emptyset \quad\Longrightarrow \quad $ there exists a model $(f,g)$ of $P$ defined over $k$.
\end{center}
\vspace{1em}
\noindent Explicitly, given a point in $\calC_P(k)$, let 
\begin{eqnarray}\label{param}
\begin{array}{lrcl}
\theta: &\quad \quad\pp^1&\to\ &\calC_P\\
&(X_0,X_1)&\mapsto& (\vartheta_1,\vartheta_2,\vartheta_3)
\end{array}
\end{eqnarray}
be the parametrization corresponding to that $k$-point. 
Then 
\begin{equation}\label{fcovariantsformula}
f= \frac{1}{\beta\mathbf{c}} \sum_{i=1}^3 \mathbf{A}_i \, \vt_i 
\end{equation}
and
\begin{equation}\label{gcovariantsformula}
g= \frac{1}{\beta^2\mathbf{c}^2}\sum_{i=1}^3 \mathbf{B}_{ij} \, \vt_i \vt_j 
\end{equation}
are forms defined over $k$ corresponding to the point $P$, i.e. with $[f,g]=P$. The factor $\beta$ depends on the choice of the parameterization $(\vt_i)$; it is calculated explicitly in the proof.

Moreover, if $P$ does not belong to the locus $M_3^{\Aut}$, the converse also holds:
\begin{center}
$\calC_P(k)\neq\emptyset \quad\Longleftrightarrow \quad $ there exists a model $(f,g)$ of $P$ defined over $k$.
\end{center} 

\end{thm}

\begin{proof}
We apply the quadratic invariants $u_i$ listed in Table \ref{tab:Covariants} to obtain typical presentations of the pair $(f,g)$. \\
 
First we show that, in the case that $\calC_P(k)$ is non-empty, the formulas do give a model over $k$. 
There exists a pair $(f_0,g_0)$ of forms with coefficients in $\bar{k}$ whose moduli point is $P$. Moreover $(f_0,g_0)$ can be chosen to have invariants equal to $\mathbf{d},\mathbf{i},\mathbf{j},\mathbf{a},\mathbf{b},\mathbf{c}$. Writing $\xi_{i,0}$ for the three quadratic covariants of the pair $(f_0,g_0)$, and applying (\ref{rf}) and (\ref{rg}) to these forms, one obtains
\begin{eqnarray}
\mathbf{c}^{-1}\sum_{i=1}^3 \mathbf{A}_i \, \xi_{i,0} &= &f_0 \label{recon1}\\
\mathbf{c}^{-2} \sum_{i=1}^3 \mathbf{B}_{ij} \, \xi_{i,0}\xi_{j,0} &= &g_0 \label{recon2}
\end{eqnarray}

Both $\vt_i$ and $\xi_{i,0}$ give parametrizations of $\calC_P$. Any two such parametrizations differ by an automorphism of $\pp^1$; so on substituting $\vt_i$ for $\xi_{i,0}$ in the lefthand side of (\ref{recon1}) and (\ref{recon2}) one obtains a new pair of forms $f_1$ and $g_1$, which have coefficients in $k$ and which differ from $f_0,g_0$ only by an automorphism of $\pp^1$. That is, there exists $M\in\GLtwo(\bar{k})$ such that $f_1=f_0\circ M$ and $g_1=g_0\circ M$. Write $M=\alpha N$, for some choice of $N\in \SLtwo(\bar{k})$ and $\alpha \in \bar{k}$. Set $\beta= \alpha^2$. Then $f_1=\beta f_0\circ N$ and $g_1=\beta^2 g_0\circ N$. With this definition of $\beta$, the pair of forms $(f,g)$ from the theorem statement are $\SLtwo(\bar{k})$-equivalent to $(f_0,g_0)$, as claimed. It remains to find $\beta$ explicitly, and to show that $\beta\in k$; then we have the desired model $(f,g)$, defined over $k$.

Write $\mathbf{d_1}, \mathbf{i_1},...$ for the invariants of the pair ($f_1,g_1$). Then from the bi-degrees of the invariants (see Table \ref{tab:Invariants}) we have  
\begin{equation}\label{degsbeta}\begin{small}
\mathbf{d_1}=\beta^2\mathbf{d}, \; \mathbf{i_1}=\beta^4\mathbf{i}, \; \mathbf{j_1}=\beta^6\mathbf{j}, \; \mathbf{a_1}=\beta^4\mathbf{a}, \; \mathbf{b_1}=\beta^6\mathbf{b}, \; \mathbf{c_1}=\beta^9\mathbf{c}
\end{small}
\end{equation}
By (\ref{relfull}) and the hypothesis $\mathbf{c}\neq0$, at least one other invariants, together with $\mathbf{c}$, does not vanish. So according to (\ref{degsbeta}) we can obtain $\beta$ explicitly as a quotient of the appropriate powers of invariants of $(f_1,g_1)$ and $(f,g)$. In particular, $\beta$ is in $k^\times$, since all the invariants are in $k$.\\


For the converse, suppose $P\notin M_3^{\Aut}$ and let $(f^\prime,g^\prime)$ be a model over $k$ corresponding to the point $P$. Note that the invariants of $(f^\prime,g^\prime)$ are in $k$. We must show $\calC_P(k)\neq\emptyset$. Write $\mathbf{d}^\prime, \mathbf{i}^\prime,$ etc. for the values of the invariants of the pair ($f^\prime,g^\prime$). Since 
$$
[\mathbf{d^\prime}:\mathbf{i^\prime}:\mathbf{j^\prime}:\mathbf{a^\prime}:\mathbf{b^\prime}:\mathbf{c^\prime}]=[\mathbf{d}:\mathbf{i}:\mathbf{j}:\mathbf{a}:\mathbf{b}:\mathbf{c}]
$$ 
as points of $\pp(2,2,3,3,4,6)$, there exists $\alpha\in \mathbb{G}_m(\bar{k})$ such that 
\begin{equation}\label{degrees}
\mathbf{I}^\prime=\alpha^{\deg(I)}\mathbf{I},
\end{equation}
for any non-zero invariant $I$. Since $P\notin M_3^{\Aut}$, at least one of $\mathbf{a}$ or $\mathbf{j}$ does not vanish (else the automorphism group contains $\zz/2\zz$); also, at least one of $\mathbf{d}$, $\mathbf{i}$ or $\mathbf{b}$ does not vanish (else the automorphism group contains $\zz/3\zz$). Therefore we can find two non-zero invariants such with coprime degrees; so (\ref{degrees}) implies $\alpha$ is in $k$. Replacing $(f^\prime,g^\prime)$ with $(\alpha^{-1}f^\prime,\alpha^{-1}g^\prime)$, we may take $(f^\prime,g^\prime)$ to have invariants exactly equal to $\mathbf{d},\mathbf{i},\mathbf{j},\mathbf{a},\mathbf{b},\mathbf{c}$. The covariants $\xi_i$ of this new pair $(f^\prime,g^\prime)$ give a parametrization of $\calC_P$ defined over $k$; so the conic has a $k$-point.   
\end{proof}

If $\mathbf{c}=0$, the conic $\calC_P$ will be singular. But when $\tilde{\mathbf{c}}\neq0$, one can choose a different pair of independent quadratic covariants: set $\tilde{u}_1=f$, $\tilde{u}_2=(f,g)_2$ and $\tilde{u}_3=(u_2,f)_1$, then proceed to form the other covariants and invariants as in Table \ref{tab:Covariants} but with $\tilde{u}_i$ in place of $u_i$. One has 
\begin{eqnarray*}
\begin{array}{llllll} 
\tilde{C}_{11} = d ,&& \quad\quad \tilde{C}_{12}  =  a ,&& \quad\quad \tilde{C}_{13}  =  0\\
\\
\tilde{C}_{22} = b+\frac{1}{3}id ,&& \quad\quad \tilde{C}_{23}  =  0 ,&& \quad\quad \tilde{C}_{33}  =  \tilde{c}
\end{array}
\end{eqnarray*}
and
\begin{eqnarray}\label{Btilde}
\begin{array}{llllll} 
\tilde{B}_{11} = a ,&& \quad \tilde{B}_{12}  =  b+\frac{1}{3}id ,&&\quad \tilde{B}_{13}  =  0\\
\\
\tilde{B}_{22} = \frac{1}{2}ia+\frac{1}{3}jd ,&&\quad \tilde{B}_{23}  =  -c ,&&\quad \tilde{B}_{33}  = \frac{1}{2}ab -\frac{1}{12}iad-\frac{1}{6}jd^2
\end{array}
\end{eqnarray}
From the covariant theory of quadratic forms (\cite{Clebsch} $\S$57), we have the following relations:
\begin{eqnarray}\label{changeov}
\tilde{\xi}_1=(\tilde{u}_2,\tilde{u}_3)_1&=&\frac{1}{2}(\tilde{C}_{22}\tilde{u}_1-\tilde{C}_{12}\tilde{u}_2)\\
\tilde{\xi}_2=(\tilde{u}_3,\tilde{u}_1)_1&=&\frac{1}{2}(\tilde{C}_{11}\tilde{u}_2-\tilde{C}_{12}\tilde{u}_1)\\
\tilde{\xi}_3=(\tilde{u}_1,\tilde{u}_2)_1&=&-\tilde{u}_3
\end{eqnarray}
Moreover
\[
\tilde{\xi}_3^2=\frac{1}{2}(\tilde{C}_{11}\tilde{u}_2^2-2\tilde{C}_{12}\tilde{u}_1\tilde{u}_2+\tilde{C}_{22}\tilde{u}_1)
\]

Notice from (\ref{Btilde}) that when $c=0$, the typical presentation of $g$ (i.e. the analogue of (\ref{recon2}) for $\tilde{\xi}_{i}$) contains no odd power of $\tilde{\xi}_3$; so we can use the above expressions to write a typical presentation of $g$ entirely in terms of the invariants $\tilde{B}_{ij}$ together with the covariants $\tilde{u}_1=f$ and $\tilde{u}_2$. 

Applying these identities and the construction of Theorem \ref{covariantstheorem}, we get an explicit solution to the FOM/FOD question that applies when $c=0$. 

\begin{thm}\label{covariantstheoremtilde}
Let $P\in M_3(k)$ be a moduli point outside of the locus $\{\tilde{c}=0\}$. The formulas (\ref{param}), (\ref{fcovariantsformula}) and (\ref{fcovariantsformula}) -- with $\mathbf{c},\, \mathbf{A}_i,\,\mathbf{B}_{ij},\,\mathbf{C}_{ij}$ everywhere replaced with $\mathbf{\tilde{c}},\, \mathbf{\tilde{A}}_i,\,\mathbf{\tilde{B}}_{ij},\,\mathbf{\tilde{C}}_{ij}$ -- yield a model $(f,g)$ of $P$. One has the folowing implications
\begin{itemize}
\item \noindent $\tilde{\calC}_P(k)\neq\emptyset \quad\Longrightarrow \quad $ there exists a model $(f,g)$ of $P$ defined over $k$.

\item If moreover $P$ does not belong to the locus $M_3^{\Aut}$, then\\
\noindent$\tilde{\calC}_P(k) \neq\emptyset \quad \Longleftrightarrow \quad $ there exists a model $(f,g)$ of $P$ defined over $k$.

\end{itemize}

\end{thm}
\begin{proof}
The proof in the case $\mathbf{c}\neq0$ is as for Theorem \ref{covariantstheorem}. In the case $\mathbf{c}=0$, the factor $\beta$ in (\ref{degsbeta}) appears to even powers only; so, using the method of the previous proof, one can recover $\beta^2$ only, and not $\beta$ itself. However, we can can show that $\beta\in k$, so using either of the two square roots of $\beta^2$ in place of $\beta$ gives us a model $(f,g)$ that is defined over $k$ and that has the required invariants.  

To see that $\beta \in k$, note that the substitution 
\[
X_1\mapsto  \tilde{\mathbf{C}}_{12}x_2-\tilde{\mathbf{C}}_{22}x_1,\quad\quad x_2\mapsto  \tilde{\mathbf{C}}_{11}x_2-\tilde{\mathbf{C}}_{12}x_1,\quad\quad x_3\mapsto x_3,
\]
corresponding to (\ref{changeov}), gives a $k$-isomorphism of $\tilde{\calC}_P$ with the conic
\[
\tilde{\calD}_P: \quad\quad \tilde{\mathbf{C}}_{11}x_2^2-2\tilde{\mathbf{C}}_{12}x_1 x_2 +\tilde{\mathbf{C}}_{22} x_1^2+2x_3^2
\]
As before, there exists a pair $(f_0,g_0)$ of forms with coefficients in $\bar{k}$ and invariants equal to $\mathbf{d},\mathbf{i},\mathbf{j},\mathbf{a},\mathbf{b},\mathbf{c}$. Write $\tilde{u}_{i,0}$ for the quadratic covariants of the pair $(f_0,g_0)$. Then the map $(X_0,X_1)\mapsto (\tilde{u}_{i,0}(X_0,X_1))_i$ gives a parameterization of $\tilde\calD_P$. The system of quadratic forms $\{\tilde{u}_{i,0}\}$ has invariants $(\tilde{u}_{i,0},\tilde{u}_{j,0})_2=\tilde{\mathbf{C}}_{ij}$.

Given a $k$-point of $\tilde{\calD}_P$, let $(t_1,t_2,t_3)\in k^3$ be a choice of coordinates for the point. Then one has a $k$-parameterization $(X_0,X_1)\mapsto (\tau_i(X_0,X_1))_i$ of $\tilde{\calD}_P$, where
\begin{eqnarray}\label{tau}
\tau_1(X_0,X_1)&=&(t_1\tilde{\mathbf{C}}_{11}+t_2\tilde{\mathbf{C}}_{12})X_0^2+2t_2 \tilde{\mathbf{C}}_{22}X_0X_1-t_1 \tilde{\mathbf{C}}_{22}X_1^2\\
\tau_2(X_0,X_1)&=&-t_2\tilde{\mathbf{C}}_{11} X_0^2+2t_1\tilde{\mathbf{C}}_{11}X_0X_1+(t_1\tilde{\mathbf{C}}_{12}+t_2\tilde{\mathbf{C}}_{22})X_1^2\\
\tau_3(X_0,X_1)&=&-t_3(\tilde{\mathbf{C}}_{11}X_0^2+\tilde{\mathbf{C}}_{12}X_0X_1+\tilde{\mathbf{C}}_{22}X_1^2)
\end{eqnarray}
Since both $\tau_i$ and $\tilde{u}_{i,0}$ are parameterizations of the same conic, there exist $N\in \SLtwo(\bar{k})$ and $\beta \in \bar{k}$, such that $\tau_i=\beta \,\tilde{u}_{i,0}\circ N$. Accordingly the invariants are related by a factor of $\beta^2$; that is, $(\tau_i,\tau_j)_2=\beta^2\tilde{\mathbf{C}}_{ij}$. On the other hand, a computation using the explicit formulas (\ref{tau}) for $\tau_i$ show that $\beta^2=4t_3^2$. In particular $\beta$ is in $k$, as claimed.

Let $\beta^\prime$ be either of the square roots of $\beta^2=4t_3^2$ in $k$. Then the required model $(f,g)$ is given by 
\begin{eqnarray}
f&=&\frac{1}{\beta^\prime}\tau_1\\
g&=&\frac{1}{\tilde{\mathbf{c}}^{2}\beta^{\prime 2}} \sum_{i=1}^3 \tilde{\mathbf{B}}_{ij} \, \tau_{i}\tau_{j}
\end{eqnarray}
\end{proof}

\subsubsection{Example}
Consider the map 
\[
\psi(x)=i\left(\frac{x-1}{x+1}\right)^3
\]
from \cite{SilvermanFOD}. The 6-tuple of invariants is $(72i, 10i ,3-3i , - 72+72i,-48  ,864i)$, which is equivalent to the point $(144,20,-12,288,192,-6912)\in M_3(\qq)$. In particular, the field of moduli is $\qq$. From the invariants it is easy to check that $\Aut(\psi)=1$. Applying the method of covariants as in Theorem \ref{covariantstheorem} to these coordinates, one obtains the conic
\begin{eqnarray*}
\quad\quad\quad\calC_\psi:\;  144 x_1^2 + 576 x_1 x_2 + 1152 x_2^2 + 384 x_1 x_3 + 768 x_2 x_3 + 1408 x_3^2=0.
\end{eqnarray*}
This can be diagonalized to 
\[ 
144X_1^2+576X_2^2+1152X_3^2 =0.
\]
Since this clearly has no points over $\rr$, the conic $\calC_\psi$ has no $\qq$-rational point; therefore Theorem \ref{covariantstheorem} implies that $\psi$ cannot be defined over $\qq$.

\subsection{Constructing maps with non-trivial automorphism group}

If $P\in M^\Aut_3(k)$, the covariants construction in the previous section can fail in two different ways:
\begin{enumerate}
\item If $P\in  \mathsf{C_2^{(1)}}(k)$, both $c=0$ and $\tilde{c}=0$, so both conics $\calC_P$ and $\tilde{\calC}_P$ are singular.
\item If $P\in \mathsf{C_2^{(2)}}(k)$ or $P\in \mathsf{C}_3(k)$, then one of the conics $\calC_P$ or $\tilde{\calC}_P$ may be non-singular, but the proofs of Theorems \ref{covariantstheorem} and \ref{covariantstheoremtilde} do not give a necessary condition for the existence of a model for $P$ defined over $k$\end{enumerate}
We shall show in sections \ref{C22} and \ref{C3} that FOM=FOD for any point $P\in \mathsf{C_2^{(2)}}$ or $P\in C_3$ for which at least one of $c$ or $\tilde{c}$ does not vanish. That leaves case (1), which is dealt with in section \ref{dep}. 

Note that in this section we refer to models by the 8-tuple of coefficients of the corresponding pair of binary forms.

\subsubsection{The locus $\mathsf{C_2^{(2)}}$}\label{C22}

If $P\in \mathsf{C_2^{(2)}}(k)$ and $\tilde{c}$ vanishes at $P$, then $c$ also vanishes at $P$, by (\ref{relfull}); so $P\in \mathsf{D_4^{(1)}}(k)$. This case is dealt with in section \ref{dep} below, so we may assume $\tilde{c}\neq 0$. A point $P\in \mathsf{C_2^{(2)}}(k)$ can be represented by coordinates $[\mathbf{d}:\mathbf{i}:0:0:\mathbf{b}:\mathbf{c}]$ with values in $k$ and with $\mathbf{d}=-2\lambda^2$, for some $\lambda\in k$ (to see this, take arbitrary coordinates for $P$ with values in $k$, then act by $\alpha=\sqrt{-2\mathbf{d}}\in \mathbb{G}_m(\bar{k})$).

Using these values, the conic $\tilde{\calC}_P$ has equation
$$
-2\lambda^2 X_1^2+\left(\frac{1}{3}\mathbf{d}\mathbf{i}+\mathbf{b}\right)X_2^2-\lambda^2 \left(\frac{1}{3}\mathbf{d}\mathbf{i}+\mathbf{b}\right) X_3^2=0
$$
This has the $k$-point $[0:\lambda:1]$. Therefore, when $\mathbf{c}\neq0$, the point $P$ always has a model defined over the field of moduli, by the construction of Theorem \ref{covariantstheoremtilde}. 

In the case $\mathbf{c}=0$, one can check that the model 
$$
\left[-2\lambda X_0 X_1, \lambda^{-3}\left(\frac{1}{3} \mathbf{d}\mathbf{i}+\mathbf{b}\right)+2\lambda X_0 X_1^3\right]
$$ 
corresponds to $P$.

\subsubsection{The locus $\mathsf{C}_3$}\label{C3}

If $P\in \mathsf{C}_3(k)$, and at least one of $a$ or $j$ does not vanish, the conic $\calC_P$ is
$$
\mathbf{a}X_0 X_1-\frac{1}{3} \mathbf{j} \mathbf{a} X_3^2,
$$
which has the  $k$-point $[1:\frac{1}{3} \mathbf{j}:1]$.
From this one obtains a model $(f,g)$ defined over $k$. Explicitly, $(f,g)$ has coefficients
$$
\left[\frac{- \mathbf{j}^{2} \mathbf{a}^{2}}{9 \mathbf{c}}, \frac{-2 \mathbf{j} \mathbf{a}^{2}}{3\mathbf{c}}, \frac{- \mathbf{a}^{2}}{\mathbf{c}}, \frac{2 \mathbf{j}^{4} \mathbf{a}^{4} + 9 \mathbf{j}^{2} \mathbf{a}^{3}}{81\mathbf{c}^{2}}, \frac{2 \mathbf{j}^{3} \mathbf{a}^{3}}{9 \mathbf{c}^{2}}, \frac{2 \mathbf{j}^{2} \mathbf{a}^{3}}{3 \mathbf{c}^{2}}, \frac{2 \mathbf{j} \mathbf{a}^{3}}{3 \mathbf{c}^{2}}, 0\right]
$$
If $\mathbf{j}=0$, i.e. $P=[0:0:0:1:0:0]$, one has a model with the following coefficients
$$
[0,0,1,0,1,0,0,1]
$$
If $\mathbf{a}=0$ then $P\in \mathsf{A}_4(k)$ -- see below.

\subsubsection{The locus $\mathsf{C_2^{(1)}}$: the case where there is no pair of independent quadratic covariants}\label{dep}

In this case $c=\tilde{c}=0$ and the covariants method yields no information. Nonetheless, in this case one can use Gr\"obner bases to find the reconstruction from the invariants and the obstruction to `FOM=FOD', as detailed in the following proposition.

\bp\label{C21}
Let $P\in \mathsf{C_2^{(1)}}(k)$ with weighted projective coordinates \mbox{$P=[\mathbf{d}:\mathbf{i}:\mathbf{j}:\mathbf{a}:\mathbf{b}:0]$}. 
\begin{enumerate}
\item If $\mathbf{d}\neq 0$, then $P$ has a model over $k$ if and only if the conic defined by
\begin{equation}\label{C21con}
9\mathbf{d}^3X^2 + 8\mathbf{d}^2Y^2 - 24\mathbf{da}YZ + (-36\mathbf{d}^2\mathbf{i} + 72\mathbf{a}^2)Z^2=0
\end{equation}
has a point, say $(x,y,z)$, over $k$. In that case set 
\begin{equation}\label{C21rels}
\begin{array}{l l}
c_5=x/z&;\quad \quad c_6=y/z\\
c_3= \mathbf{d}/2 &;\quad \quad c_4=\frac{2\mathbf{a}}{\mathbf{d}^2}-\frac{1}{3 \mathbf{d}} c_6.
\end{array}
\end{equation}
Then the model over $k$ is $[1,0,c_3,c_4,c_5,c_6,-c_3c_5,c_3^2c_4]$.

\item If $\mathbf{d}= 0$, then $P\in \mathsf{D_4^{(2)}}(k)$. Assume neither $\mathbf{i}$ nor $\mathbf{j}$ vanishes. In this case, $P$ always has a model over $k$, given by $[0,0,0,-27\mathbf{i}^3,-27\mathbf{i}^3,0,24\mathbf{j}^2,0]$. 

\item If $\mathbf{d}=\mathbf{j}=0$, then $P\in \mathsf{D}_8(k)$, and $P$ always has a model over $k$, given by $[0,0,0,1,0,0,0,1]$. 

\item If $\mathbf{d}=\mathbf{i}=0$, then $P\in \mathsf{A}_4(k)$. Over $\bar{k}$ one has the model $\psi$ with coefficients $[0,0,0,1,0,2\sqrt{-3},0,1]$. The obstruction $I_k(\phi)$ contains a class corresponding to the conic
\[
X^2+3Y^2-2Z^2=0
\]
\end{enumerate}
\ep

We need the following two lemmas.

\bl \cite{Beau}
Let $B\leq \pgltwo(k)$ be a cyclic subgroup of order two. Then $B$ is conjugate by an element of $  \pgltwo(k)$ to $A_\alpha:=\langle z \mapsto \alpha/z\rangle$ for some $\alpha \in k^\times/(k^\times)^2$.
\el

\bl
If $\phi=(f,g)$ is a model defined over $k$ for a point $P\in \mathsf{C_2^{(1)}}(k)$ with $d\neq0$, then each element of $\Aut(\phi)$ is defined over $k$.
\el

\bpf
From the normal form in Table \ref{tab:Auts}, one can see that fixed points of the involution $\gamma \in \Aut(\phi)$ are precisely the roots of $f$. Since $f$ has coefficients in $k$, the element $\gamma$ must also be defined over $k$.
\epf

\begin{proof}[Proof of Proposition \ref{C21}]
For part (1): suppose that $\mathbf{d}\neq 0$ and that $P$ has a model $\phi$ over $k$. Then by the preceding two lemmas, we may assume that $\Aut(\phi)$ is generated by $\langle z \mapsto \alpha/z\rangle$ for some $\alpha \in k^\times/(k^\times)^2$. Computing relative invariants for this group as in Section \ref{autos}, one sees that $\phi$ must have the form 
\[
( X_0^2-\alpha X_1^2,  c_4(X_0^4+\alpha^2 X_1^4)+c_5(X_0^3X_1+\alpha X_0 X_1^3)+c_6 X_0^2X_1^2)
\]
for some $c_4,c_5,c_6\in k$. Set $c_3=-\alpha$. By computing a Gr\"obner basis, one can see that the only non-trivial relations between the coefficients $c_i$ and the invariants are those given in equations (\ref{C21con}) and (\ref{C21rels}). This establishes the assertion of the proposition.

One can also construct the conic (\ref{C21con}) by explicit cohomology computations, without using the two lemmas. Start with the model $\psi$ defined over $L=k(\sqrt{-2\mathbf{d}})$ with coefficients 
\[
[0,\sqrt{-2\mathbf{d}},0,1,0,-3\mathbf{a}/\mathbf{d},0,\mathbf{i}/2-3\mathbf{a}^2/4\mathbf{d}^2];
\]
this can easily be found using Gr\"obner bases. Set $e=\mathbf{i}/2-3\mathbf{a}^2/4\mathbf{d}^2$. The model has $A=\Aut(\psi)=\langle z\mapsto -z\rangle$. And $N(A)=\mathfrak{D}_\infty=\mathbb{G}_m\rtimes \mu_2$. The class $c_\psi$ is represented by the cocycle
\[
\gamma_\tau =
\left\{
	\begin{array}{ll}
		1  & \mbox{if } \tau \in G_L \\
		(z\mapsto r/z) & \text{otherwise}
	\end{array}
\right.
\]
Write $G_{L/k}=\langle \sigma \rangle$. By inflation-restiction we can reduce from $G_k$ to $G_{L/k}$, where we have the following situation (see \cite{SilvermanFOD}: proof of Theorem 3.2):
\begin{displaymath} 
\xymatrix{ 
\tiny{\matrixofmap} & H^1(G_{L/k}, L^\times\rtimes\mu_2(L))\ar@{->}[r]^{p}\ar@{->}[d]^{\visom} & H^1(G_{L/k}, L^\times\rtimes\mu_2(L))\ar@{->}[d]^{\visom} \\ 
x\ar@{|->}[u]   & \dfrac{k^\times/\mathrm{Nm}_{L/k}(L^\times)}{x=x^{-1}}\ar@{->}[r]^{x\mapsto x^2}&  \dfrac{k^\times/\mathrm{Nm}_{L/k}(L^\times)}{x=x^{-1}}
 }
\end{displaymath} 
Since $\frac{k^\times/\mathrm{Nm}_{L/k}(L^\times)}{x=x^{-1}}$ is a group of exponent 2, the element $c_\psi$ is in the image of $p$ if and only if it is trivial; that is, if and only either $r$ is the norm of an element of $L$ or $\tau=1$; that is, if and only there is a solution in $k$ to $X^2+2dY^2-rZ^2$. One can easily check that this conic is isomorphic over $k$ to the one in the theorem statement.

For part (4): The model $\psi$ was computed in Section \ref{autos}: it is the normal form of a map with automorphism group $A=\langle z\mapsto-z, z\mapsto \frac{1}{z}, z\mapsto i\frac{z+1}{z-1}\rangle$; we have $N(A)=A\rtimes \langle\frac{z+1}{z-1}\rangle\simeq \frak{S}_4$, and $Q\simeq\{\pm1\}$. Let $L=k(\sqrt{-3})$ and write $G_{L/k}=\langle \sigma \rangle$. From the model $\psi$, we compute that 
\[
\gamma_\tau =
\left\{
	\begin{array}{ll}
		1  & \mbox{if } \tau \in G_L \\
		-1 & \text{otherwise}
	\end{array}
\right.
\]
is a cocycle representing $c_\psi$. Using the splitting coming from the presentation of $A$ as a semidirect product we can find an inverse image for $c_\psi$ under $p$; namely  
\[
\eta_\tau =
\left\{
	\begin{array}{ll}
		1  & \mbox{if } \tau \in G_L \\
		z\mapsto \frac{z+1}{z-1} & \text{otherwise}
	\end{array}
\right.
\]
Computing the twist corresponding to this $\pgltwo$-cocycle, we obtain the conic from the statement.

The other cases are easily verified by computing the invariants of the claimed models, which come from the normal forms given in Section \ref{autos}. 

\epf

\section{Even degrees: Quadratic Maps}\label{evendegs}

Covariants can used to explicitly construct generic maps of any even degree from their moduli.  
 
In this section we work through the covariants method in the case of quadratic maps in order to illustrate the general method for even degree maps and to make clear how the invariant theory approach to $M_d$ tallies with the description in \cite{SilvermanSpace} and \cite{ManesYasufuku}.

\subsection{The invariants and covariants}
 \begin{table}[h]
  \centering
\begin{center}
    \begin{tabular}{ c  c  l  c c c}
    \hline
    \multicolumn{3}{c}{Covariants}& Order  & Degree  \\ \hline
    $H$ & $=$ & $(g,g)_2$ & 2 & 2&\\
    $t$ & $=$ & $(g,H)_1$ & 3 & 3&\\
    $s_1$ & $=$ & $(g,f^3)_3$ & 0 & 4 &\\
    $s_2$ & $=$ & $(H,f^2)_2$ & 0 & 4 &\\
    $s_3$ & $=$ & $(t,g)_3$ & 0 & 4 &\\  
    $R$ & $=$ & $(t,f^3)_3$ & 0 & 6 &\\
    $V_0$ & $=$ & $(H,f)_1$ & 1 & 3&\\ 
    $V_1$ & $=$ & $(g,f^2)_2$ & 1& 3& \\ 
	$b_0$ & $=$ & $(V_1,f)_1$ & 0& 4& \\
	$b_1$ & $=$ & $(V_0,f)_1$ & 0& 4& \\
	$r$ & $=$ & $(b_0,b_1)_1$ & 0 & 6 &\\
    $a_{ijk}$ & $=$ & $(g,V_i)_1(g,V_j)_1(g,V_k)_1$ & 0& 10& for $0\leq i,j,k\leq 1$\\ 
   
    \hline
    \end{tabular}
\end{center}
\caption{Covariants for quadratic maps}
  \label{tab:Covariantsdegree2}
\end{table} 

Given a quadratic rational map $\phi=F_0/F_1$, one obtains a pair of binary forms -- one cubic form $g$ and one linear form $f$ -- using the Clebsch-Gordan isomorphism. Then one can construct the covariants by transvection as described in \cite{GraceYoung} ($\S$ 138). The covariants that we shall need are given in Table \ref{tab:Covariantsdegree2}. The ring of invariants is generated by $s_1,s_2,s_3$ and $R$. Since $R$ is the only invariant of degree 6, $r$ must be a multiple of $R$; in fact $R=r$. 

One usually uses $\sigma_1$ and $\sigma_2$ -- the first and second elementary symmetric functions in the multipliers of the map -- as the coordinates on $M_2$. If one writes 
\begin{eqnarray*}
\sigma_i=\tau_i\cdot \rho^{-1} \quad\quad\quad \mathrm{for\quad} i=1,2 
\end{eqnarray*}
where $\rho $ is the resultant of $F_0$ and $F_1$, then one can relate the invariants from Table \ref{tab:Covariantsdegree2} to $\sigma_1$ and $\sigma_2$ as follows:
\begin{eqnarray*}
s_1&=&   5 \tau_1 + 2 \tau_2 + 6 \rho \\
s_2&=& \frac{ 2}{3} \tau_1 + \frac{2}{3} \tau_2 - 4 \rho\\
s_3&=& \frac{-4}{27} \tau_1 +\frac{ 2}{27} \tau_2 + \frac{2}{9} \rho\\
\end{eqnarray*}

There is a single relation among the four basic invariants:

\begin{eqnarray}\label{deg2rel}
r^2 - \frac{1}{2}s_1^2 s_3 + \frac{1}{2}s_2^3=0
\end{eqnarray}

Let $A$ be the ring of invariants graded by degree. Then the fourth Veronese subring $A^{[4]}$ is generated by $s_1,s_2,s_3$ and $r^2$. But (\ref{deg2rel}) expresses $r^2$ as a polynomial in the other invariants, so
\begin{eqnarray*}
\Proj \,A \cong \Proj \,A^{[4]} = \Proj \,k[s_1,s_2,s_3]= \Proj\, k[\tau_1,\tau_2,\rho] =\pp (4,4,4)\cong\pp^2
\end{eqnarray*}
Thus one recovers the usual description of $M_2\cong \affine^2$ as the complement of the line $\{\rho=0\}$ inside $\Proj \,k[\tau_1,\tau_2,\rho] =\pp^2$.

Let us write the relation (\ref{deg2rel}) in terms of $\tau_1,\tau_2$ and $\rho$:
\begin{eqnarray*}
r^2= -2 \tau_1^3 - \tau_1^2 \tau_2 + \tau_1^2 \rho + 8 \tau_1 \tau_2 \rho - 12 \tau_1 \rho^2 + 4 \tau_2^2 \rho -
    12 \tau_2 \rho^2 + 36 \rho^3\end{eqnarray*}
Comparing with \cite{SilvermanMod} Prop. 4.15, one can see that the locus $\{r=0\}$ is precisely the locus in $M_2$ of maps with non-trivial automorphism group.

\subsection{Constructing a model from the moduli}

The method of this section will work for maps of any even degree. For concreteness, we illustrate it only in the quadratic case.

We first note that $a_{ijk}$ are invariant under permutations of the indices. In terms of the basic invariants they are
\[
a_{000}=\frac{1}{9}s_3 r;\quad\quad a_{100}=0;\quad\quad a_{110}=\frac{-1}{9}s_2 r;\quad\quad a_{111}=\frac{-2}{9}s_1 r
\]

\begin{thm} Let $P\in M_2(k)$ be a moduli point corresponding to a $\bar{k}$-conjugacy class of maps with trivial automorphism group.  Let $s_1,s_2,s_3,r\in k$ be weighted projective coordinates with values in $k$ corresponding to $P$ (see Lemma \ref{hilb90}). Let $W_0,W_1\in k[X_0,X_1]$ be any pair of linear forms for which the corresponding coordinate transformation
$$
X_0\mapsto W_0(X_0,X_1)\quad; \quad X_1\mapsto W_1(X_0,X_1)
$$
is in $\SLtwo (k)$. Then the linear and cubic forms 
\begin{eqnarray*}\label{recoverquadratic}
 f&=&b_0W_0-b_1 W_1\\
g&=&\frac{9}{2r} \sum_{i,j,k\in\ff_2} (-1)^{i+j+k} a_{ijk}W_{i+1} W_{j+1} W_{k+1}\\
\end{eqnarray*}
have coefficients in $k$ and correspond to the moduli point $P$, i.e. $[f,g]=P$. 
\end{thm}
\begin{proof}
We apply the the typical presentation for systems of odd order forms from section \ref{oddorder}.  

Note that $b_i,a_{ijk}\in k$, since $s_1,s_2,s_3,r\in k$. There exists a pair $(f_0,g_0)$ of forms with coefficients in $\bar{k}$ whose invariants are equal to the $s_1,s_2,s_3$ and $r$ from the theorem statement. Applying (\ref{linrel}) and (\ref{genrel}) to a the linear form $f_0$ and cubic form $g_0$, and using the notation of Table \ref{tab:Covariantsdegree2}, one has
\begin{eqnarray*}\label{recoverquadratic}
r f_0&=&b_0 V_0-b_1 V_1\\
2r^3 g_0&=&9 \sum_{i,j,k\in\ff_2} (-1)^{i+j+k} a_{ijk}V_{i+1} V_{j+1} V_{k+1}\\
\end{eqnarray*}
Note that $V_i$ are covariants of the pair $(f_0,g_0)$. Since the point $P$ is outside of the locus of maps with nontrivial automorphism group, $r$ is nonzero.

Given any pair $W_0,W_1\in k[X_0,X_1]$ of linearly independent linear forms, let $M\in\GLtwo(\bar{k})$ be the coordinate transformation from $V_0,V_1$, to $W_0,W_1$. Write $M=\alpha N$ for $N\in \SLtwo(\bar{k})$ and $\alpha \in \bar{k}$. Note that $r=\det[V_0,V_1]=\alpha^{-2}$. Then
\begin{eqnarray*}
r f_1&=&b_0W_0-b_1 W_1\\
2r^3 g_1&=&9 \sum_{i,j,k\in\ff_2} (-1)^{i+j+k} a_{ijk}W_{i+1} W_{j+1} W_{k+1}\\
\end{eqnarray*}
On the other hand, $f_1=\alpha^{-1}f_0\circ N$ and $g_1=\alpha^{-3} g_0\circ N$. Setting $f=f_1$ and $g=rg_1$ one obtains the forms a pair of forms projectively equivalent to those from the theorem statement; their invariants are $r^2 s_i$, so they correspond to the moduli point $P$, as required.
\end{proof}

\ti{Acknowledgements.} The author would like to thank Bart van Steirtegem for pointing out the route of describing $M_3$ in terms of the ring of $\SLtwo$-invariants and for many helpful discussions. The author is also grateful to Nikita Miasnikov and Barinder Banwait for useful discussions.

\section{Appendix: Invariants for cubic rational maps}

The expressions for the invariants terms of the coefficients $c_i$ of the pair 
\begin{align*} 
 f&=
c_{1} X_0^{2} + c_{2} X_0 X_1 + c_{3} X_1^{2}\\
g&=
c_{4} X_0^{4} + c_{5} X_0^{3} X_1 + c_{6} X_0^{2} X_1^{2} + c_{7} X_0 X_1^{3} + c_{8} X_1^{4}
 \end{align*}
 of quadratic and quartic forms are
\begin{align*}
  d&=
-\frac{1}{2} c_{2}^{2} + 2 c_{1} c_{3}\\
  i&=
\frac{1}{6} c_{6}^{2} - \frac{1}{2} c_{5} c_{7} + 2 c_{4} c_{8}\\
  j&=
-\frac{1}{36} c_{6}^{3} + \frac{1}{8} c_{5} c_{6} c_{7} - \frac{3}{8} c_{4} c_{7}^{2} - \frac{3}{8} c_{5}^{2} c_{8} + c_{4} c_{6} c_{8}\\
  a&=
c_{3}^{2} c_{4} - \frac{1}{2} c_{2} c_{3} c_{5} + \frac{1}{6} c_{2}^{2} c_{6} + \frac{1}{3} c_{1} c_{3} c_{6} - \frac{1}{2} c_{1} c_{2} c_{7} + c_{1}^{2} c_{8}\\
  b&=
-\frac{1}{8} c_{3}^{2} c_{5}^{2} + \frac{1}{3} c_{3}^{2} c_{4} c_{6} + \frac{1}{12} c_{2} c_{3} c_{5} c_{6} - \frac{1}{36} c_{2}^{2} c_{6}^{2} - \frac{1}{18} c_{1} c_{3} c_{6}^{2} - \frac{1}{2} c_{2} c_{3} c_{4} c_{7} \\
&\quad + \frac{1}{24} c_{2}^{2} c_{5} c_{7} + \frac{1}{12} c_{1} c_{3} c_{5} c_{7} + \frac{1}{12} c_{1} c_{2} c_{6} c_{7} \\
&\quad - \frac{1}{8} c_{1}^{2} c_{7}^{2} + \frac{1}{3} c_{2}^{2} c_{4} c_{8} + \frac{2}{3} c_{1} c_{3} c_{4} c_{8} - \frac{1}{2} c_{1} c_{2} c_{5} c_{8} + \frac{1}{3} c_{1}^{2} c_{6} c_{8} \\
 c&= \frac{1}{32} c_{3}^{3} c_{5}^{3} - \frac{1}{8} c_{3}^{3} c_{4} c_{5} c_{6} - \frac{1}{32} c_{2} c_{3}^{2} c_{5}^{2} c_{6} + \frac{1}{8} c_{2} c_{3}^{2} c_{4} c_{6}^{2} + \frac{1}{4} c_{3}^{3} c_{4}^{2} c_{7} \\
 &\quad- \frac{1}{16} c_{2} c_{3}^{2} c_{4} c_{5} c_{7} + \frac{1}{32} c_{2}^{2} c_{3} c_{5}^{2} c_{7} + \frac{1}{32} c_{1} c_{3}^{2} c_{5}^{2} c_{7} - \frac{1}{8} c_{2}^{2} c_{3} c_{4} c_{6} c_{7} - \frac{1}{8} c_{1} c_{3}^{2} c_{4} c_{6} c_{7}\\
 &\quad + \frac{1}{32} c_{2}^{3} c_{4} c_{7}^{2} + \frac{3}{16} c_{1} c_{2} c_{3} c_{4} c_{7}^{2} - \frac{1}{32} c_{1} c_{2}^{2} c_{5} c_{7}^{2}  - \frac{1}{32} c_{1}^{2} c_{3} c_{5} c_{7}^{2} + \frac{1}{32} c_{1}^{2} c_{2} c_{6} c_{7}^{2} \\
&\quad - \frac{1}{32} c_{1}^{3} c_{7}^{3} - \frac{1}{2} c_{2} c_{3}^{2} c_{4}^{2} c_{8} + \frac{1}{4} c_{2}^{2} c_{3} c_{4} c_{5} c_{8} + \frac{1}{4} c_{1} c_{3}^{2} c_{4} c_{5} c_{8} - \frac{1}{32} c_{2}^{3} c_{5}^{2} c_{8}\\
  &\quad- \frac{3}{16} c_{1} c_{2} c_{3} c_{5}^{2} c_{8} + \frac{1}{8} c_{1} c_{2}^{2} c_{5} c_{6} c_{8} + \frac{1}{8} c_{1}^{2} c_{3} c_{5} c_{6} c_{8} - \frac{1}{8} c_{1}^{2} c_{2} c_{6}^{2} c_{8} - \frac{1}{4} c_{1} c_{2}^{2} c_{4} c_{7} c_{8} \\
  &\quad- \frac{1}{4} c_{1}^{2} c_{3} c_{4} c_{7} c_{8} + \frac{1}{16} c_{1}^{2} c_{2} c_{5} c_{7} c_{8} + \frac{1}{8} c_{1}^{3} c_{6} c_{7} c_{8} + \frac{1}{2} c_{1}^{2} c_{2} c_{4} c_{8}^{2} - \frac{1}{4} c_{1}^{3} c_{5} c_{8}^{2}
 \end{align*}

\newpage




\bibliographystyle{amsalpha}
\bibliography{bibcon}

\end{document}